\documentclass[10pt,letterpaper]{amsart}
\usepackage[utf8]{inputenc}
\usepackage[english]{babel}
\usepackage{subfigure}
\usepackage{amsmath}
\usepackage{amsthm} 
\usepackage{hyperref} 
\usepackage{amsfonts}
\usepackage{tikz-cd}
\usepackage{amssymb}
\usepackage{csquotes}
\usepackage[
backend=biber,
style=alphabetic,
sorting=nty,doi=false,isbn=false, url=false
]{biblatex}

\newtheorem*{theorem*}{Theorem}
\newtheorem{theorem}{Theorem}[subsection]
\newtheorem{prop}[theorem]{Proposition}
\newtheorem{defi}[theorem]{Definition}
\newtheorem{lemma}[theorem]{Lemma}
\newtheorem{corollary}[theorem]{Corollary}
\newtheorem{conj}[theorem]{Conjecture}

\newtheorem*{corollary*}{Corollary}
\newtheorem*{prop*}{Proposition}

\newcommand{\C}{\mathbb{C}}
\newcommand{\Z}{\mathbb{Z}}
\newcommand{\Pro}{\mathbb{P}}
\newcommand{\Q}{\mathbb{Q}}

\newcommand{\eff}{\text{eff}}
\newcommand{\et}{\text{\'et}}
\newcommand{\ch}{\text{Chow}}
\newcommand{\chef}{\text{Chow}^\eff}

\newcommand{\ho}{\text{Hom}}

\newcommand{\CH}{\text{CH}}

\theoremstyle{remark}
\newtheorem{remark}[theorem]{Remark}
\tikzcdset{
  cells={font=\everymath\expandafter{\the\everymath\displaystyle}},
}

\newcommand{\info}{{
  \bigskip
  \footnotesize

  \textsc{Institut de Math\'ematiques de Bourgogne, UMR 5584 CNRS, Universit\'e Bourgogne, F-21000 Dijon, France}\par\nopagebreak
  \textit{E-mail address}: \texttt{ivan-alejandro.rosas-soto@u-bourgogne.fr}
  }}

\usepackage[left=3cm,right=3.5cm,top=2cm,bottom=3.2cm]{geometry}
\author{Iv\'an Rosas Soto}
\date{December 2023}
\title{Hodge structures through an \'etale motivic point of view}
\begin{document}
\begin{abstract}
We define the category of \'etale Chow motives as the \'etale analogue of the Grothendieck motives and prove that it embeds in the triangulated category of \'etale motives, introduced by Cisinski and D\'eglise, which we denote by $\text{DM}_\et(k)$. Using similar techniques to the ones used by Rosenschon and Srinivas we prove a refined version of the equivalence between the rational Hodge conjecture and the integral \'etale analogue. From these results we deduce an equivalence between the generalized Hodge conjecture with rational coefficients and an \'etale analogue with integral coefficients.
\end{abstract}
\maketitle

\tableofcontents
\section{Introduction}

Let $X$ be a smooth projective variety over $\C$, $k \in \mathbb{N}$ and consider the cycle class map $c^k: \CH^k(X)\to H^{2k}_B(X,\Z(k))$ where $\Z(k)=(2\pi i)^k \Z$ whose image is a subgroup of the Hodge classes $\text{Hdg}^{2k}(X,\Z)$. The integral Hodge conjecture asks whether or not this map is surjective. Putting $n=\text{dim}_\C(X)$, then for $k=0$ and $k=n$ the conjecture is immediately true and also for $k=1$ by the Lefschetz (1,1) theorem, but for $k=2$ the statement is not true as is shown by the counterexamples given by Atiyah and Hirzebruch in \cite{AH} (a torsion class which is not algebraic) and by K\'ollar in \cite{BCC} (a non-algebraic non-torsion class) respectively; nevertheless, with rational coefficients the validity of the statement regarding the surjectivity of the cycle class map is still an open question and it is known as the Hodge conjecture. In a more general and ambitious framework, there exists another conjecture, called the \textit{generalized Hodge conjecture}, which deals with sub-Hodge structures of smooth projective varieties in different weights and levels. To be more precise the conjecture for weight $k$ and level $k-2c$ (or equivalently for weight $k$ and coniveau $c$) says that for any rational sub-Hodge structure $H\subset H^k(X,\Q)$ of level at most $k-2c$ there exists a closed subvariety $Z \hookrightarrow X$ of codimension $\geq c$ such that 
\begin{align*}
    H \subset \text{im}\left\{ H^{k-2c}(\widetilde{Z},\Q(-c)) \xrightarrow{\gamma_*}H^k(X,\Q) \right\}
\end{align*}
where $\gamma_*=i_*\circ d_*$, $i_*$ is the Gysin map associated to the inclusion $i:Y\hookrightarrow X$ and $d:\widetilde{Z}\to Z$ is a resolution of singularities.

Although the use of rational coefficients in the statement of both conjectures is necessary, in recent years Rosenschon and Srinivas proved that the validity of the Hodge conjecture with rational coefficients is equivalent to an \'etale version of it with integral coefficients, using Lichtenbaum cohomology groups, denoted by $\CH^k_L(X)$, as presented in \cite[Theorem 1.1]{RS}. For this, they constructed a L-cycle class map $c^k_L:\CH^k_L(X) \to H^{2k}_B(X,\Z(k))$ such that $\text{im}(c_L^k)\subset \text{Hdg}^{2k}(X,\Z)$. This new approach gives that the restriction $\CH^k_L(X)_\text{tors}\to H^{2k}_B(X,\Z(k))_\text{tors}$ is surjective in an unconditional way. 

In the present article we study an \'etale (or Lichtenbaum) version of the generalized Hodge conjecture using Lichtenbaum cohomology groups, and ask whether or not it is possible to give an equivalence between the generalized Hodge conjecture with $\Q$-coefficients and its \'etale integral version. With this purpose we revisited the equivalence regarding the validity of the Hodge conjecture given in \cite{RS}. 

We start by looking at the two different definitions for the \'etale analogue of Chow groups: Using the triangulated category of motives as in \cite{CD16} and Lichtenbaum cohomology groups as in \cite{RS}. For the first case, we set the functorial behaviour of the \'etale Chow groups with respect to projective morphisms. These functorial properties allows us to define correspondences in the \'etale setting and a composition law as in the classical case. After that we present some of the principal properties of Lichtenbaum cohomology and we make the link between the two definitions of \'etale Chow group. With the the properties that we give we construct the category of \'etale Chow motives, denoted by $\ch_\et(k)$, fits in the following commutative diagram:
 \[
  \begin{tikzcd}
\ch(k)^{op} \arrow{d}\arrow{r}{\Phi}&  \text{DM}(k) \arrow{d}\\
\ch_\et(k)^{op}  \arrow{r}{\Phi^\et} & \text{DM}_{\et}(k)
  \end{tikzcd}
\]
where $\ch(k)$ is the category of Chow motives, $ \text{DM}(k)$ and $ \text{DM}_{\et}(k)$ are the triangulated categories of motives over $k$ and its \'etale counterpart respectively, and the horizontal arrows are full embedding. This construction gives a similar construction to the category defined by Kahn in \cite{kahn2002}. 

After that we present the following refined version of \cite[Theorem 1.1]{RS}:

\begin{prop*}[see Proposition \ref{propE}]
Let $X$ be a complex smooth projective variety and consider a sub-Hodge structure $W \subset H^{2k}_B(X,\Z(k))$ of type $(k,k)$. Then $W$ is $L$-algebraic, i.e. $W \subset \text{im}(c^k_L)$, if and only if $W\otimes \Q$ is algebraic.
\end{prop*}

Starting from \cite[Remark 5.1.a]{RS}, we use the \'etale analogue of the generalized Hodge conjecture given there in order to study the classical version. In Proposition \ref{k} we give a complete proof of the equivalence between the different versions of the generalized Hodge conjecture (usual case and Lichtenbaum) in weight $2k-1$ and level 1, result that was stated in the same remark in \cite{RS}. For that, we split the proof in two parts: in the first, we prove that the L-generalized Hodge conjecture in weight $2k-1$ and level 1 is equivalent to the fact that a part of the Hodge conjecture for the product of $X\times C$ is true for all smooth and projective curve $C$, after that we invoke Proposition \ref{propE}. To finalize, our main results are the following:

 First, we obtain a characterization of the generalized conjecture (for all $X\in \text{SmProj}_\C$) given in \cite{RS} which follows the idea of the classical case, that is, in term of realization of motives previously defined in section 2 and the Hodge conjecture:
 
 \begin{theorem*}[see Theorem \ref{teo}]
The Lichtenbaum generalized Hodge conjecture  for all $X \in \text{SmProj}_\C$ holds if and only if the following two conditions hold: 
\begin{itemize}
    \item the Lichtenbaum Hodge conjecture holds,
    \item a homological \'etale motive is effective if and only if its Hodge realization is effective.
\end{itemize}
\end{theorem*}

With this, we obtain as a corollary the following equivalence: 

\begin{corollary*}[see Corollary \ref{corf}]
The generalized Hodge conjecture with $\Q$-coefficients holds if and only if the generalized integral L-Hodge conjecture holds.
\end{corollary*}

The article is organized as follows: In the first two subsections of section 2 we recall the different definitions that we can use as the \'etale analogue of Chow groups, using the category of \'etale motives as in \cite{CD16} and Lichtenbaum cohomology groups, stating the similarities between the two definitions. In subsections 1.3 and 1.4 we present the correspondences with their classical operation and the construction of the category of Chow \'etale motives as in the case of Chow motives, for instance see \cite{MNP}.

In sub-section 3.1 we mainly focus on the Hodge conjecture and its generalized version. We start by some reminders and definitions of classical Hodge theory. For the complex case we give the proof of the generalized Roitman theorem and with this result we show that if we restrict to a sub-Hodge structure $W \subset H^{2k}_B(X,\Z(k))$ and ask whether or not $W\otimes \Q$ is algebraic in the usual sense, it is equivalent to ask if $W$ is L-algebraic. 

With respect to the generalized Hodge conjecture, in sub-section 3.3 we show several equivalences between the classical case and the L-version (involving Lichtenbaum cohomology and integral Hodge structures) in different weights and levels using characterizations through the Hodge conjecture (\'etale and classical setting) and the effectiveness of motives. In the last subsection we mention the consequences of the equivalence between the classical and \'etale version of the generalized Hodge conjecture regarding Bardelli's example in \cite{Bar}.

\section*{Conventions}
For a field $k$ we denote the n-dimensional $k$-projective space as $\Pro^n_k$ and $\text{SmProj}_k$ is the category of smooth and projective reduced $k$-schemes. Let $G$ be an abelian group, $\ell$ a prime number and $r\geq 1$, then we denote  $G[\ell^r]:=\left\{g \in G \ | \ \ell^r\cdot g =0 \right\}$, $G\{\ell\}:=\bigcup_{r} G[\ell^r]$, $G_\text{tors}$ denotes the torsion sub-group of $G$ and $G_{\text{free}}:=G/G_{\text{tors}}$ its torsion free quotient. The prefix ``L-" indicates the respective version of some result, conjecture, group, etc. in the Lichtenbaum setting. $H_B^i(X,\Z(n))$ denotes the Betti cohomology groups of $X$. 

\section*{Acknowledgements}

The author thanks his advisors Fr\'ed\'eric D\'eglise and Johannes Nagel for their suggestions, useful discussions and the time for reading this article. This work is supported by the EIPHI Graduate School (contract ANR-17-EURE-0002) and the FEDER/EUR-EiPhi Project EITAG. We thank the French “Investissements d’Avenir” project ISITE-BFC (ANR-15-IDEX-0008)
 and the French ANR project “HQ-DIAG” (ANR-21-CE40-0015).

\section{\'Etale Chow groups and Lichtenbaum cohomology}
We have an \'etale analogue of the Chow groups, which leads us to the construction of the category of \'etale Chow motives with integral coefficients. In section 1.1 we recall the definition of the \'etale Chow groups and present correspondences in this setting. It should be noted that the category that we construct cannot be defined as the subcategory of $\text{DM}_\et(k)$ generated by elements of pure weight 0, in the sense of Bondarko, contrary to the Nisnevich case, for the definition of weight structure see \cite[Section 1]{Bon} and \cite[Theorem 2.1.1]{Bon}. For a detailed explanation of the non-existence of the weight structure see \cite[Remark 7.2.26]{CD16}.

\subsection{\'Etale Chow groups}

In this subsection we use the category of \'etale motives, since we do not mention much more details about the construction and/or functorial behaviour of the category, for further details about these properties we refer the reader to \cite{Ayo} and \cite{CD16}. Let $k$ be a field and let $R$ be a commutative ring. We denote the category of effective motivic \'etale sheaves with coefficients in $R$ over the field $k$ as $\text{DM}^{\eff}_\et(k,R)$. If we invert the Lefschetz motive, we then obtain the category of motivic \'etale sheaves denoted by $\text{DM}_\et(k,R)$. If no confusion arises we denote $\text{DM}^{(\eff)}_\et(k):=\text{DM}^{(\eff)}_\et(k,\Z)$. 
One defines the \textbf{\'etale motivic cohomology} group of bi-degree $(m,n)$ with coefficients in a commutative ring $R$ as
\begin{align*}
H_{M,\et}^m(X,R(n)):= \ho_{\text{DM}_\et(k,R)}(M_\et(X),R(n)[m]).
\end{align*}
where $M_\et(X)=\rho^* M(X)$ with $\rho$ is the canonical map associated to the change of topology $\rho:\left(\text{Sm}_k\right)_\et\to \left(\text{Sm}_k\right)_{\text{Nis}}$ which induces an adjunction $\rho^*:= \mathbf{L}\rho^*:\text{DM}(k)\rightleftarrows  \text{DM}_\et(k):\mathbf{R}\rho_*=:\rho_*$. In particular we define the \textbf{\'etale Chow groups} of codimension $n$ as the \'etale motivic cohomology in bi-degree $(2n,n)$ with coefficients in $\Z$, i.e. 
\begin{align*}
\text{CH}_\et^n(X):&=H_{M,\et}^{2n}(X,\Z(n))\\
&=\ho_{\text{DM}_\et(k)}(M_\et(X),\Z(n)[2n]).
\end{align*}

\begin{remark}
    Let $k$ be a field and let $\ell$ be a prime number different from the characteristic of $k$. By the rigidity theorem for torsion motives, see \cite[Theorem 4.5.2]{CD16}, we have an isomorphism 
    $$H_{M,\et}^m(X,\Z/\ell^r(n))\simeq H_\et^m(X,\mu_{\ell^r}^{\otimes n}).$$
\end{remark}

\subsubsection{Gysin morphism and functoriality properties}

With respect to functoriality properties of the \'etale Chow groups we should mention that we can recover well-known properties analogous to that of classical Chow groups, such as pull-back and proper pushforwards of cycles. In particular, we get a degree map. All these properties will arise from the properties of the category $\text{DM}_\et(k)$ (resp. $\text{DM}(k)$) and the covariant functor $M_\et(-)$ (resp. $M(-)$). 

Let us recall that the canonical map $\rho:\left(\text{Sm}_k\right)_\et\to \left(\text{Sm}_k\right)_{\text{Nis}}$ induces an adjunction of triangulated categories
\begin{align*}
\rho^*: \text{DM}_{gm}(k)\rightleftarrows \text{DM}_{gm,\et}(k):\rho_*,
\end{align*}
which leads us to express the \'etale Chow groups in terms of morphism in the category $\text{DM}(k)$ as follows
\begin{align*}
H^m_{M,\et}(X,\Z(n))&:=\text{Hom}_{\text{DM}_\et(k)} (M_\et(X),\Z_\et(n)[m])\\
&\simeq  \text{Hom}_{\text{DM}(k)} (M(X),\rho_*\Z_\et(n)[m]).
\end{align*} 

\begin{prop}
The comparison map 
\begin{align*}
\sigma^{m,n}:H^m_M(X,\Z(n)) \to H^m_{M,\et}(X,\Z(n))
\end{align*}
coming from the adjunction of triangulated categories, is compatible with pullbacks, pushforward and intersection products.
\end{prop}

\begin{proof}
Consider the adjunction of triangulated categories
\begin{align*}
\rho^*: \text{DM}(k)\rightleftarrows \text{DM}_\et(k):\rho_*
\end{align*}
where $\rho^*$ is the functor induced by the \'etale sheafification and $\rho_*$ is the right adjoint which is a forgetful functor which forgets that the complexes are \'etale. As we have said, the cycle class map is obtained by the following use of the adjunction
\begin{align*}
\text{Hom}_{ \text{DM}(k)}\left(M(X),\Z(n)[m]\right)\xrightarrow{\rho^*}\text{Hom}_{ \text{DM}_\et(k)}(\rho^* M(X),\rho^*\Z(n)[m])
\end{align*}
where $\Z(n)$ is the motivic complex of twist n and $M(X)$ is the  triangulated motive associated with $X$. By adjunction we have
\begin{align*}
\text{Hom}_{ \text{DM}_\et(k)}(\rho^* M(X),\rho^*\Z(n)[m]) \simeq \text{Hom}_{ \text{DM}(k)}(M(X),\rho_*\rho^*\Z(n)[m])
\end{align*}
so we obtain a canonical map $\Z(n)\to \rho_* \rho^* \Z(n)=\rho_* \Z_\et(n)$ given by the unit transformation associated to the adjunction. Now, the functorial properties of maps $f:X\to Y$ come from the (covariant) functorial properties of the motive $M(X)$ and the existence of Gysin maps, for more details about the existence of Gysin morphisms we refer to \cite{DegI} and \cite{DegII}. To be more precise: Let $f:Y\to X$ be a morphism of relative dimension $d$, then we have induced commutative squares 
\[
  \begin{tikzcd}
   M(X)(d)[2d]\arrow{r}{f^!} \arrow{d}{\rho^*} & M(Y) \arrow{d}{\rho^*} & & & M(Y) \arrow{r}{f_*} \arrow{d}{\rho^*} & M(X) \arrow{d}{\rho^*}\\
   M_\et(X)(d)[2d] \arrow{r}{f^!} & M_\et(Y) & & & M_\et(Y) \arrow{r}{f_*} & M_\et(X)
  \end{tikzcd}
\]
which induce the pullback and pushforward for proper morphisms. In fact, any morphism of motivic complexes like the one given by the adjunction will yield a morphism of cohomology theory compatible both with pullbacks and pushforward. 

Finally, we need to prove the compatibility with respect to products. This property comes from the fact that we have a quasi-isomorphism
\begin{align*}
\Z(i)\otimes \Z(j) \xrightarrow{\sim} \Z(i+j)
\end{align*}
and that the functor $\rho^*$ is monoidal, i.e. $\rho^*(M\otimes N) \simeq \rho^*(M)\otimes \rho^*(N)$, therefore we also obtain that 
\begin{align*}
\Z(i)_\et\otimes \Z(j)_\et \xrightarrow{\sim} \Z(i+j)_\et
\end{align*}
For intersection products the remaining part is to consider that the product comes from the operation $\alpha \cdot \beta =\Delta^*(\alpha \otimes \beta)$.
\end{proof}

\subsection{Lichtenbaum cohomology groups}

We consider a second notion of the \'etale version of Chow groups, which is the well known Lichtenbaum cohomology groups, groups defined by the hypercohomology of the \'etale sheafification of Bloch's complex sheaf. These groups are characterized by Rosenschon and Srinivas in \cite{RS} using \'etale hypercoverings. In this context, we consider $\text{Sm}_k$ as the category of smooth separated $k-$schemes over a field $k$. We denote $z^n(X,\bullet)$ the cycle complex of abelian groups defined by Bloch 
\begin{align*}
   z^n(X,\bullet): \cdots \to z^n(X,i) \to \cdots \to z^n(X,1)\to z^n(X,0) \to 0 
\end{align*}
where the differentials are given by the alternating sum of the pull-backs of the face maps and whose homology groups define the higher Chow groups $\text{CH}^n(X,2n-m)=H_m(z^n(X,\bullet))$.

Let us recall that  $z^n(X,i)$ and the complex $z^n(X,\bullet)$ is covariant functorial for proper maps and contravariant functorial for flat morphisms between smooth $k$-schemes, see \cite[Proposition 1.3]{Blo}, therefore for a topology $t \in \left\{\text{flat}, \ \et, \ \text{Nis},  \ \text{Zar} \right\}$ we have a complex of $t$-presheaves
$z^n(-,\bullet):U \mapsto z^n(U,\bullet)$. In particular the presheaf $z^n(-,i):U \mapsto z^n(U,i)$ is a sheaf for $t \in \left\{\text{flat}, \ \et, \ \text{Nis},  \ \text{Zar} \right\}$, see \cite[Lemma 3.1]{Ge04}, and then $z^n(-,\bullet)$ is a complex of sheaves for the small \'etale, Nisnevich and Zariski sites of $X$. We set the complex of $t$-sheaves 
\begin{align*}
R_X(n)_t = \left(z^n(-,\bullet)_t \otimes R \right)[-2n]
\end{align*}
where $R$ is an abelian group and for our purposes we just consider $t= \text{Zar}$ or $\et$ and then we compute the hypercohomology groups $\mathbb{H}^m_t(X,R_X(n)_t)$. For example, setting $t=$ Zar and $R=\Z$ the hypercohomology of the complex allows us to recover the higher Chow groups $\text{CH}^n(X,2n-m)\simeq \mathbb{H}_{\text{Zar}}^m(X,\Z(n))$. We denote the motivic and Lichtenbaum cohomology groups with coefficients in $R$ as
\begin{align*}
H_M^m(X,R(n))= \mathbb{H}_\text{Zar}^m(X,R(n)), \quad H_L^m(X,R(n))=\mathbb{H}_\et^m(X,R(n))
\end{align*}
and in particular we set $\text{CH}_L^n(X)=H^{2n}_L(X,\Z(n))$. Let $\pi: X_\et\to X_{\text{Zar}}$ be the canonical morphism of sites, then the associated adjunction formula $\Z_X(n)\to R \pi_* \pi^* \Z_X(n)= R \pi_* \Z_X(n)_\et$ induces \textit{comparison morphisms}
\begin{align*}
\text{H}_M^m(X,\Z(n)) \xrightarrow{\kappa^{m,n}} \text{H}_L^m(X,\Z(n))
\end{align*}
for all bi-degrees $(m,n) \in \Z^2$. We can say more about the comparison map, due to \cite[Theorem 6.18]{VV}, the comparison map  $\kappa^{m,n}:H^m_M(X,\Z(n))\to H^m_L(X,\Z(n))$ is an isomorphism for $m\leq n+1$ and a monomorphism for $m \leq n+2$.

In some cases it is possible to obtain more information about the Lichtenbaum cohomology groups and the comparison map between them and higher Chow groups. For instance there is a quasi-isomorphism $A_{X}(0)_\et=A$, the latter as an \'etale sheaf, thus we obtain that the Lichtenbaum cohomology agrees with the usual \'etale cohomology, i.e.
$\text{H}_L^m(X,A(0)) \simeq H^n_\et(X,A)$ for all $m \in \Z_{\geq 0}$ and in particular  $\text{CH}^0_L(X)=\Z^{\pi_0(X)}$.  In the next step, $n=1$, since there is a quasi-isomorphism of complexes $\Z_X(1)_\et\sim \mathbb{G}_m[-1]$ we obtain the following isomorphisms 
\begin{align*}
    \text{CH}^1(X)&\simeq \text{CH}_L^1(X)=\text{Pic}(X)\\
    \text{H}_L^3(X,\Z(1)) &\simeq  \text{H}_\et^3(X,\mathbb{G}_m[-1])  = \text{Br}(X)
\end{align*}
where $\text{Pic}(X)$ and $\text{Br}(X)$ are the Picard and Grothendieck-Brauer groups of $X$ respectively. In fact bi-degree $(n,1)$ by \cite[Corollary 3.4.3]{V} there exists an isomorphism $\text{H}_M^n(X,\Z(1)) \simeq  \text{H}_\text{Zar}^{n-1}(X,\mathbb{G}_m)$ because the quasi-isomorphism $\Z_X(1)\sim \mathbb{G}_m[-1]$ also holds in Zariski topology. As a particular case consider $\text{H}_M^3(X,\Z(1)) \simeq  \text{H}_\text{Zar}^2(X,\mathbb{G}_m)=0$ because $\text{H}_M^m(X,\Z(n))=0$ if $m>2n$ where the Grothendieck-Brauer group of $X$ is not always zero (for instance consider $X$ an Enriques surface).

In bi-degree $(4,2)$ the comparison map is known to be injective but in general not surjective; we have a short exact sequence
\begin{align*}
0 \to \text{CH}^2(X)\xrightarrow{\kappa^2}\text{CH}^2_L(X)\to \text{H}^3_{\text{nr}}(X,\Q/\Z(2))\to 0,
\end{align*}
where $\mathcal{H}^3_\et(\Q/\Z(2))$ is the Zariski sheaf associated to $U\mapsto \text{H}_\et(U,\Q/\Z(2))$ and unramified part is the global section $\text{H}_\text{nr}^3(X,\Q/\Z(2))=\Gamma(X,\mathcal{H}^3_\et(\Q/\Z(2)))$, for a proof we refer to \cite[Proposition 2.9]{Kahn}. If $k=\C$ the latter group surjects onto the torsion of the obstruction, in codimension 4, to the integral Hodge conjecture, i.e.
\begin{align*}
    \text{H}_\text{nr}^3(X,\Q/\Z(2)) \twoheadrightarrow \left(\text{Hdg}^4(X,\Z)/\text{im}\left\{c^2:\CH^2(X)\to H^4_B(X,\Z(2))\right\}\right)_\text{tors}
\end{align*}
and then in general is not zero and the comparison map $\kappa^2$ is not surjective, for more details see  \cite[Th\'eor\`eme 3.7]{CTV}.

\begin{remark}
The adjunction formula for rational coefficients, the morphism
$\Q_X(n)\to R \pi_* \Q_X(n)_\et$ turns out to be an isomorphism (see {\cite[Th\'eor\`eme 2.6]{Kahn}}), thus $\text{H}_M^m(X,\Q(n)) \simeq \text{H}_L^m(X,\Q(n))$ for all $(m,n)\in \Z^2$.
\end{remark}

If $R$ is torsion then we can compute the Lichtenbaum cohomology as an \'etale cohomology. To be more precise for a prime number $\ell$, $r\in \mathbb{N}\geq 1$ and $R=\Z/\ell^r$ then we have the following quasi-isomorphisms
\begin{align*}
    (\Z/\ell^r)_X(n)_\et \xrightarrow{\sim} \begin{cases}
      \mu_{\ell^r}^{\otimes n} &\text{if char}(k)\neq \ell \\
      \nu_r(n)[-n] &\text{if char}(k)= \ell
    \end{cases}
\end{align*}
where $\nu_r(n)$ is the logarithmic Rham-Witt sheaf. After passing to direct limit we have also quasi-isomorphisms
\begin{align*}
    (\Q_\ell/\Z_\ell)_X(n)_\et \xrightarrow{\sim} \begin{cases}
      \varinjlim_{r}\mu_{\ell^r}^{\otimes n} &\text{if char}(k)\neq \ell \\
      \varinjlim_{r}\nu_r(n)[-n] &\text{if char}(k)= \ell
    \end{cases}
\end{align*}
and finally set $ (\Q/\Z)_X(n)_\et =\bigoplus (\Q_\ell/\Z_\ell)_X(n)_\et\xrightarrow{\sim } \Q/\Z(n)_\et$. The following result is a well known fact, known as the \textit{Suslin rigidity theorem}, about the morphism $\Z_X(n)\to R\pi_* \Z_X(n)_\et$ for $n\geq \text{dim}(X)$ over $k=\bar{k}$.

\begin{prop}{\cite[Theorem 4.2]{V},\cite[Section 2]{Geis}}\label{lemGe}
Let $X$ be a smooth projective variety of dimension $d$ over an algebraically closed field $k$. Then for $n\geq d$ the canonical map $\pi:X_\et \to X_\text{Zar}$ the induced morphism between complexes of Zariski sheaves $\Z_X(n)\to R\pi_* \Z_X(n)_\et$ is a quasi-isomorphism.
\end{prop}

\begin{proof}
Having that $\Q_X(n)\to R\pi_*\Q_X(n)_\et$ is a quasi-isomorphism for all $n\in \mathbb{N}$, then we only have to focus on torsion coefficients. In characteristic zero this was already proved by Suslin in \cite[Prop. 4.1, Thm. 4.2]{V}, an in general away from the characteristic of the field $k$. 

For the general case, assume that $n=d$ and $l \in\mathbb{N}$. If $k$ has positive characteristic then by \cite[Lemma 2.4]{Ge10} for a constructible sheaf $\mathcal{F}$ we have that $R\ho(\mathcal{F},\Z/l(d)[2d])[-1]\cong R\ho(\mathcal{F},\Z(d)[2d])$ and also there exists a perfect pairing of finite groups $\text{Ext}^{1-m}(\mathcal{F},\Z_X(d)[2d])\times H^m_c(X_\et,\Z/l))\to \Q/\Z$ so this gives us an isomorphism $H^{2d-m}_M(X,\Z/l(d))^*\simeq H^m_c(X_\et,\Z/l)$. Since $X$ is smooth, Poincar\'e duality holds for \'etale cohomology, see \cite[Chapter VI \S 11]{Mil80}, then $H^m_c(X_\et,\Z/l))^*\simeq H^{2d-m}_\et(X,\Z/l(d))$ and therefore we obtain the isomorphisms $H^{2d-m}_M(X,\Z/l(d))\simeq H^{2d-m}_\et(X,\Z/l(d)) \simeq H^{2d-m}_L(X,\Z/l(d))$. As in \cite[Theorem. 4.2]{V} for a general $n\geq d$ we use the homotopy invariance of the higher Chow groups
\begin{align*}
    H^{2d-m}_M(X,\Z/l(d)) &\simeq H^{2d-m}_M(X\times \mathbb{A}_k^n,\Z/l(d))\\ 
    &\simeq H^m_c(X\times \mathbb{A}_k^n,\Z/l)^* \\
    &\simeq H^{m-2(n-d)}_c(X,\Z/l(d-n))^*.
\end{align*}
To conclude, we have a quasi-isomorphism $\left(\Z/l\right)_X(n) \to R\pi_* \left(\Z/l\right)_X(n)_\et$ for all $l \in \mathbb{N}$ therefore $\left(\Q/\Z\right)_X(n) \to R\pi_* \left(\Q/\Z\right)_X(n)_\et$ as well. Thus from the commutative diagram 
\begin{equation*}
  \begin{tikzcd}[column sep=2em]
   \arrow{r} &H^{m-1}_M(X,\Q/\Z(n))\arrow{d}{\simeq}  \arrow{r} & H^m_M(X,\Z(n)) \arrow{r} \arrow{d}{} & H^m_M(X,\Q(n)) \arrow{r}  \arrow{d}{\simeq } & H^m_M(X,\Q/\Z(n)) \arrow{d}{\simeq} \to  \\
  \arrow{r} &H^{m-1}_L(X,\Q/\Z(n))  \arrow{r} & H^m_L(X,\Z(n)) \arrow{r} & H^m_L(X,\Q(n))\arrow{r}   & H^m_L(X,\Q/\Z(n)) \to
  \end{tikzcd}
\end{equation*}
we then conclude that $H^m_M(X,\Z(n))\simeq H^m_L(X,\Z(n))$.
\end{proof}

\subsection{Category of \'etale Chow motives}

\subsubsection{Correspondences}
Proceeding in a similar way as in the case of pure motives, we need to introduce the concept of correspondences which have a main role in the definition of the morphisms in the category of \'etale Chow motives. For this construction we will use the notion of \'etale Chow groups, but always keeping in mind the following: Consider a field $k$ of characteristic exponent $p$ and a smooth $k$-scheme $X$ which is of finite type. For all bi-degree $(m,n) \in \Z^2$ there exists a map $    \rho^{m,n}_X: H^m_L(X,\Z(n)) \to H^m_{M,\et}(X,\Z(n))$ which is induced by the $\mathbb{A}^1$-localization functor of effective \'etale motivic sheaves. If we tensor by $\Z[1/p]$, then $\rho^{m,n}_X$ becomes an isomorphism, for a proof we refer to \cite[Theorem 7.1.2]{CD16}. In particular the two definitions coincide in characteristic zero.

\begin{defi}
Let $X$ and $Y$ be smooth projective varieties. An \'etale correspondence from $X$ to $Y$ of degree $r$ is defined as follows: if $X$ is purely of codimension $d$
\begin{align*}
\text{Corr}_\et^r(X,Y)=\text{CH}_\et^{r+d}(X\times Y).
\end{align*}
For the general case
\begin{align*}
\text{Corr}_\et^r(X,Y)= \bigoplus_{i=1}^n \text{CH}_\et^{r+d_i}(X_i\times Y)
\end{align*}
where $X=\coprod_{i=1}^n X_i$ and $d_i$ is the dimension of $X_i$.
\end{defi}

For $\alpha \in \text{Corr}_\et^r(X,Y)$ and $ \beta \in \text{Corr}_\et^s(Y,Z)$ we define the composition $\beta \circ \alpha \in \text{Corr}_\et^{r+s}(X,Z)$ of correspondences as 
\begin{align*}
\beta \circ \alpha = \left(\text{pr}_{13}\right)_*\left(\text{pr}^*_{12}\alpha \cdot \text{pr}^*_{23}\beta\right)
\end{align*}
where $\text{pr}_{12}:X\times_k  Y \times_k Z \to X \times_k Y$ (similar definition for $\text{pr}_{23}$ and $\text{pr}_{13}$ with the respective change in the projection's components). 
\begin{prop}
The composition of correspondences is an associative operation.
\end{prop}

\begin{proof}

 To see that this operation is associative, we recall Gysin morphism for \'etale motives. Consider $X$, $Y$ and $S$ smooth schemes over $k$ such that there exists a cartesian square of smooth schemes
    \begin{equation}
\begin{tikzcd}
X \times_S Y \arrow{r}{q} \arrow{d}{g}& Y \arrow{d}{f} \\
X \arrow{r}{p} & S
\end{tikzcd}
\end{equation} 
 with $p$ and $q$ are projective morphism and $\text{dim}(X/S)=\text{dim}(X\times_S Y/Y)$, thus by \cite[Proposition 5.17]{DegII} we have the following commutative diagrams
      \begin{equation}
\begin{tikzcd}
M(X \times_S Y)(-n)[-2n]  \arrow{d}{g_*}& M(Y) \arrow{l}{q^*} \arrow{d}{f_*} & & \CH_\et^{i+n}(X \times_S Y) \arrow{r}{q_*}& \CH_\et^{i}(Y) \\
M(X)(-n)[-2n]  & \arrow{l}{p^*} M(S) & & \CH_\et^{i+n}(X) \arrow{u}{g^*}\arrow{r}{p_*}& \CH_\et^{i}(S) \arrow{u}{f^*}
\end{tikzcd}
\end{equation} 
where $n=\text{dim}(X/S)$.\label{rem1} 

Consider the following commutative diagram
\[
\begin{tikzcd}
X\times Y \times Z \times W \arrow{r}{\text{pr}^{XYZW}_{XYZ}} \arrow{d}{\text{pr}^{XYZW}_{XZW}}& X\times Y \times Z \arrow{d}{\text{pr}^{XYZ}_{XZ}} \\
X\times Z \times W \arrow{r}{\text{pr}^{XZW}_{XZ}} & X\times Z
\end{tikzcd}
\]
by (2) we have that $\left(\text{pr}^{XYZW}_{XYZ} \right)_*\left(\text{pr}^{XYZW}_{XZW}\right)^*= \left(\text{pr}^{XYZ}_{XZ}\right)^* \left( \text{pr}^{XZW}_{XZ}\right)_*$, so the rest of the proof is similar to the proof of \cite[16.1.1.(a)]{Fulton}.
\end{proof}

\begin{remark}
\begin{itemize}
    \item The composition of correspondences on $\text{Corr}_\et(X,X)$ induces a binary operation which gives to it a ring structure. In general it is not a commutative operation.
    \item Let $X$ be a smooth projective scheme of dimension n, then the \'etale cycle $\Delta^\et_X$, induced by the diagonal, is the identity for the composition operation, i.e. for $\alpha \in \text{Corr}_\et^r(X,Y) $ and $\beta \in \text{Corr}_\et^r(Y,X)$ we obtain that $\alpha \circ \Delta^\et_X = \alpha$ and $\Delta^\et_X\circ\beta=\beta$.
\end{itemize}
\end{remark}

\subsubsection{Operation of correspondences}

It is possible to define the addition and product of correspondences in the following way: suppose that $\alpha \in \text{Corr}_\et(X,X)$ and $\beta \in \text{Corr}_\et(Y,Y)$, then we define the element $\alpha + \beta$ as the element resulting from the following operation on cycles:
\begin{align*}
 \text{CH}_{\et}(X\times X)\oplus \text{CH}_{\et}(Y\times Y) &\hookrightarrow \text{CH}_{\et}\left(\left(X\coprod Y\right)\times\left(X\coprod Y\right)\right)\\
(\alpha,\beta) &\mapsto (i_1)_*\alpha + (i_2)_*\beta
\end{align*}
where $i_1:X\times X \hookrightarrow \left(X \coprod Y\right)\times \left(X \coprod Y\right)$ is the usual closed immersion map (similar definition for $i_2$ and $Y$). In a similar way we define the tensor product of correspondences as
\begin{align*}
\text{CH}_{\et}(X\times X)\otimes \text{CH}_{\et}(Y\times Y) &\to \text{CH}_\et(X\times Y \times X \times Y)\\
(\alpha,\beta) &\mapsto \text{pr}_{XX}^* \alpha \cdot \text{pr}_{YY}^* \beta
\end{align*}
where $\text{pr}_{XX}:X\times Y \times X \times Y \to X\times X$, similar definition for $\text{pr}_{YY}$. Both of this structures will play a big role for the definition of operations in the category of Chow \'etale motives. Another important operation is the transposition of cycles

\begin{defi}
Let $X$ and $Y$ be smooth projective varieties and let $\tau:X\times Y \to Y\times X$ which permutes the components $(x,y)\mapsto(y,x)$. Let $\Gamma \in \text{CH}_\et^n(X\times Y)$, we define the transpose cycle as $\Gamma^t := \tau_*(\Gamma)$.
\end{defi}

\subsubsection{Action on cycles and cohomology groups}

Let $X$ and $Y$ be smooth projective varieties. For a correspondence $\Gamma \in \text{Corr}_\et^r(X,Y)$ we define the action $\Gamma_*:\text{CH}^i_\et(X)\to \text{CH}^{i+r}_\et(Y)$ as
\begin{align*}
\Gamma_* Z =\text{pr}_{Y*}\left(\Gamma \cdot \text{pr}^*_X(Z)\right) \in \text{CH}^{i+r}_\et(Y)
\end{align*}
for $Z \in \text{CH}^i_\et(X)$. Here arises the necessity of working with \'etale Chow groups, because of its functoriality properties for proper maps, instead of Lichtenbaum cohomology. In order to use an action considering Lichtenbaum cohomology, it would be necessary to invert the exponential characteristic of the base field. 


Let $\Gamma  \in \text{Corr}_L^r(X,Y)$ be a correspondence of degree $r$, then we have an operation on Betti (with integral coefficients) and $\ell$-adic cohomology $\Gamma_*:H^i(X)\to H^{i+2r}(Y)$ defined as in the following expression
\begin{align*}
\Gamma_* z:= \text{pr}_{Y*}\left(c_L(\Gamma)\cup \text{pr}_X^*(z)\right)
\end{align*}
with $z \in H^i(X)$. As we will see in the following sections and subsections this notion of actions will be the cornerstone to have a well defined Hodge conjecture and generalized Hodge conjecture in the Lichtenbaum setting.

\subsection{\'Etale Chow motives}

Let $\text{SmProj}_k$ be the category of smooth projective varieties over $k$. We construct the category of effective Lichtenbaum motives over $k$, denoted by $\chef_\et(k)$, as follows:
\begin{itemize}
\item The elements are tuples $(X,p)$ where $X$ is a smooth projective variety and $p \in \text{Corr}_\et^0(X,X)$ is an idempotent element, i.e. $p\circ p =0$.
\item Morphism $(X,p)\to (Y,q)$ are the elements of the form $f=q \circ g \circ p$ where $g \in \text{Corr}_\et^0(X,Y)$, therefore
\begin{align*}
\text{Hom}_{\chef_\et(k)}\left((X,p),(Y,q)\right) = q \circ \text{Corr}_\et^0(X,Y) \circ p
\end{align*}
\end{itemize}

Finally, the category $\ch_\et(k)$ of Chow \'etale motives is defined in the following way: the objects are triplets $(X,p,m)$ where $X$ is a smooth projective variety, $p$ is a correspondence of degree 0 and idempotent and $m \in \Z$. The morphisms $(X,p,m)\to (Y,q,n)$ are defined as
\begin{align*}
\text{Hom}_{\ch_\et(k)}\left((X,p,m),(Y,q,n)\right) = q \circ \text{Corr}_\et^{n-m}(X,Y) \circ p
\end{align*}

As in the theory of Chow motives, for \'etale motives there is an obvious fully faithful functor $\ch_\et^{\eff}(k)\hookrightarrow\ch_\et(k)$.

We define a functor $h_\et:\text{SmProj}_k^{\text{op}}\to \ch_\et(k)$ as
\begin{align*}
h_\et:\text{SmProj}_k &\to \ch_\et(k)\\
X&\mapsto h_\et(X):=(X,\text{id}_X,0)\\
\left( X\xrightarrow{f}Y \right) &\mapsto \left(h_\et(Y)\xrightarrow{h_\et(f)}h_\et(X) \right) 
\end{align*}
where $\text{id}_X$ is the element that acts as the identity on the correspondences from $X$ to itself and $h_\et(f)=\kappa([\Gamma_f^t])$. 

\begin{remark}
Let us remark that there was another construction of a category of Chow étale motives, which here we denote $\ch_\et^K(k)$, given in \cite[\S 5]{kahn2002} by Kahn. This category is pseudo-abelian and rigid symmetric monoidal. The definition of such (effective) category is similar as the one we gave, but just considering elements $(X,p)$ where $p=p^2 \in \text{Corr}_\et(X,X)\otimes \Q$ and morphisms between $(X,p)$ and $(Y,q)$ are correspondences $f \in \text{Corr}_\et(X,Y)$ such that $f \otimes \Q = q \circ \tilde{f}\circ p =\tilde{f} \in \text{Corr}_\et(X,Y)\otimes \Q$.
\end{remark} 

\begin{prop}
Similar to the theory of pure Chow motives, there exists a fully faithful embedding functor $F:\ch_\et(k)^{\text{op}}\hookrightarrow \text{DM}_\et(k)$
\end{prop}
\begin{proof}
Let $X,Y,Z \in \text{SmProj}_k$. The map $\epsilon_{X,Y}: \text{Hom}_{\text{DM}_\et(k)}(M(X),M(Y))\xrightarrow{\simeq}\text{Corr}_{\et}^0(Y,X)$ is an isomorphism, which can be obtained with the same arguments as in \cite[Proposition 20.1]{MVW}. The compatibility with respect composition is obtained as in \cite[Theorem 3.17]{Jin} using \cite[Proposition 2.39]{Jin}.
\end{proof}

\section{Hodge and Generalized Hodge conjecture}
\subsection{Hodge conjecture and Lichtenbaum cohomology}

Fix an integer $k \in \Z$. An integral pure Hodge structure $H$ of weight $k$ is a finitely generated $\Z$-module such that $H\otimes \C=\bigoplus_{p+q=k} H^{p,q}$ where $H^{p,q}$ is a complex vector space with $H^{q,p}=\overline{H^{p,q}}$. For $m \in \Z$ we denote as $\Z(m)$ the Tate Hodge structure of weight $-2m$ whose Hodge decomposition is concentrated in bi-degree $(-m,-m)$. For a  pure Hodge structure $H$ of weight $k$ its Tate twist $H(m)$ is defined to be the tensor product $H \otimes_\Z \Z(m)$ which is a Hodge structure of weight $k-2m$ and its decomposition is
\begin{align*}
    H(m)\otimes_\Z\C= \bigoplus_{p+q=k-2m}H(m)^{p,q}=\bigoplus_{p+q=k-2m}H^{p-m,q-m}
\end{align*}

If $X$ is a complex smooth projective variety we denote by $\text{Hdg}^{2n}(X,\Z)$ the Hodge classes of $X$ of weight $2n$, defined as 
\begin{align*}
    \text{Hdg}^{2n}(X,\Z) := \left\{ \alpha \in H^{2n}_B(X,\Z(n)) \ \Big| \ \rho(\alpha) \in F^n H^{2n}(X,\C)\right\}
\end{align*}
where $\rho: H^{2n}_B(X,\Z)\to  H^{2n}(X,\C)$ and $F^p  H^{2n}(X,\C)= \bigoplus_{i \geq p} H^{i,2n-i}(X)$. Note that by definition $H^{2n}_B(X,\Z)_\text{tors} \subset \text{Hdg}^{2n}(X,\Z)$.
The image of the cycle class map to Betti cohomology $c^n:\CH^n(X) \to H^{2n}_B(X,\Z(n))$ is contained in $\text{Hdg}^{2n}(X,\Z)$.  We denote as $\text{HC}^n(X)$  the following statement:

\begin{conj}[Hodge conjecture with integral coefficients]
For a complex smooth projective variety $X$ and $n \in \mathbb{N}$ the image of the cycle class map $c^n:\CH^n(X)\to H^{2n}_B(X,\Z(n))$ is $\text{Hdg}^{2n}(X,\Z)$.
\end{conj}

If we replace in the conjecture $\Z$ by $\Q$, we will denote as $\text{HC}^n(X)_\Q$. For $n\neq 1$ it is known that the Hodge conjecture (with integral coefficients) does not hold, even if we work with torsion free classes. We define the obstruction to the integral Hodge conjecture as $Z^{2i}(X):=\text{Hdg}^{2i}(X,\Z(i))/\text{im}(c^i)$. In  \cite{AH} it is proved that for every prime number $p$ there exists a smooth variety $X$ such that $Z^{4}(X)[p]\neq 0$.

This conjecture can be stated in terms of motives as well, using the Hodge realization to characterize the validity of the conjecture for the category $\text{SmProj}_\C$
\begin{prop}
Consider $k=\C$ and $\rho_H$ the Hodge realization for $\ch(\C)$ (with rational coefficients), then HC$(X)_\Q$ holds for all $X\in \text{SmProj}$ if and only if $\rho_H$ is a full functor.
\end{prop}

Before going into the proof of the equivalences of the weaker version of the equivalence between the Hodge conjecture with rational coefficients and the Lichtenbaum Hodge conjecture let us recall the definitions of Deligne cohomology and intermediate Jacobians and some facts about them. 
Fix an integer $k\geq 0$ and let us denote the $k-$th intermediate Jacobian $J^k(X)$ as the complex torus
\begin{align*}
J^k(X):=H^{2k-1}(X,\C)/(F^kH^{2k-1}(X,\C) \oplus H^{2k-1}(X,\Z)).    
\end{align*}
Consider the Deligne complex $\Z(p)_D$ of a complex manifold $X$ defined as
\begin{align*}
    0 \to \Z(p)\to \mathcal{O}_X\to \Omega^1_X \to \ldots \to  \Omega^{p-1}_X \to 0.
\end{align*}
We then define the \textbf{Deligne cohomology groups} as the hypercohomology groups of the Deligne complex i.e.
\begin{align*}
    H_D^k(X,\Z(p)):=\mathbb{H}^k_\text{an}(X,\Z(p)_D).
\end{align*}
which yields an exact sequence relating Hodge classes and intermediate Jacobians
\begin{align*}
    0 \to  J^k(X) \to H^{2k}_D(X,\Z(k)) \to  \text{Hdg}^{2k}(X,\Z) \to 0.
\end{align*}

\begin{remark}
The definition of intermediate Jacobians can be extended to pure Hodge structures of odd weight. Assume that $H$ is a Hodge structure of weight $2k-1$ then we define the complex torus $J^{k}(H):=H_\C/(F^k H\oplus H)$. This construction is functorial with respect to morphisms of Hodge structures. For more details about these facts see \cite[Remarque 12.3]{Voi} and \cite[Section 3.5]{PS}. 
\end{remark}

There exist maps $c^k_D:\CH^k(X) \to H^{2k}_D(X,\Z(k))$ and $\Phi_X^k: \CH^k(X)_\text{hom} \to J^k(X)$ called the \textbf{Deligne cycle class} and the \textbf{Abel-Jacobi map} respectively. There is a useful relation between the Deligne cycle class map, the Abel-Jacobi map and the cycle class map given by the following commutative diagram with exact rows:
  \[
  \begin{tikzcd}
  0 \arrow{r} & \text{CH}^{k}(X)_\text{hom} \arrow{r} \arrow{d}{\Phi_X^k} &\text{CH}^k(X) \arrow{r}{c^k}  \arrow{d}{c^k_{D}} & I^k(X)  \arrow{d}{\text{into}}\arrow{r} & 0\\
  0 \arrow{r} &J^k(X) \arrow{r} &H^{2k}_D(X,\Z(k)) \arrow{r}   & \text{Hdg}^{2k}(X,\Z) \arrow{r}& 0.
  \end{tikzcd}
\]

For Lichtenbaum cohomology groups we have analogous maps, $c^k_{L,D}:\CH^k_L(X) \to H^{2k}_D(X,\Z(k))$ and $\Phi_{X,L}^k: \CH^k_L(X)_\text{hom} \to J^k(X)$ (the construction of the first one is done in \cite[Theorem 4.4]{RS}) which fit in a similar commutative diagram as the one given before.

\begin{remark}\label{remDel}
Let $\ell$ be a prime number and $r \in \mathbb{N}$. Notice that the exact triangle $0\to \Omega^{\leq n-1}[-1] \to \Z_D(n)\to \Z(n)\to 0$ induces maps $c^{m,n}_{D,B}:H^{m}_D(X,\Z(n))\to H^m_B(X,\Z(n))$ which fit in the following commutative diagram
\begin{equation*}
  \begin{tikzcd}
 0 \arrow{r}& H^{m-1}_D(X,\Z(n))\otimes \Z/\ell^r \arrow{d}\arrow{r} &H^{m-1}_D(X,\Z/\ell^r(n)) \arrow{r}{\beta_D} \arrow{d}{\simeq}& H^{m}_D(X,\Z(n))[\ell^r] \arrow{d}{c^{m,n}_{D,B}} \arrow{r}&0 \\
0 \arrow{r} & H^{m-1}_B(X,\Z(n))\otimes \Z/\ell^r \arrow{r}& H^{m-1}_B(X,\Z/\ell^r(n)) \arrow{r}{\beta} & H^{m}_B(X,\Z(n))[\ell^r]\arrow{r}& 0 
  \end{tikzcd}
\end{equation*}
where $\beta_D$ is the morphism induced by the exact triangle $0\to\Z_D(n) \xrightarrow{\cdot \ell^r} \Z_D(n)\to (\Z/\ell^r)_D(n)\to 0$. Also we obtain another commutative diagram:
\begin{equation*}
  \begin{tikzcd}
  0 \arrow{r} & H^{m-1}_L(X,\Z(n))\otimes \Z/\ell^r \arrow{r} \arrow{d}{} & H^{m-1}_\et(X,\mu_{\ell^r}^{\otimes n}) \arrow{r}  \arrow{d}{\simeq } & H^m_L(X,\Z(n))[\ell^r] \arrow{d}{c_{D,L}^{m,n}}\arrow{r} & 0\\
  0 \arrow{r} & H^{m-1}_D(X,\Z(n))\otimes \Z/\ell^r \arrow{r} \arrow{d}& H^{m-1}_D(X,\Z/\ell^r(n)) \arrow{r}{\beta_D} \arrow{d}{\simeq}  & H^m_D(X,\Z(n))[\ell^r] \arrow{r} \arrow{d}{c_{D,B}^{m,n}} & 0 \\
    0 \arrow{r}& H^{m-1}_B(X,\Z(n))\otimes \Z/\ell^r \arrow{r}& H^{m-1}_B(X,\Z/\ell^r(n)) \arrow{r}{\beta}   & H^m_B(X,\Z(n))[\ell^r] \arrow{r} & 0.
  \end{tikzcd}
\end{equation*}
By the snake lemma the arrows
\begin{align*}
    H^{m-1}_L(X,\Z(n))/\ell^r \to H^{m-1}_D(X,\Z(n))/\ell^r \text{ and  } H^{m-1}_D(X,\Z(n))/\ell^r \to H^{m-1}_B(X,\Z(n))/\ell^r
\end{align*}
 are injective while the arrows 
 \begin{align*}
       H^m_L(X,\Z(n))[\ell^r] \to H^m_D(X,\Z(n))[\ell^r] \text{ and } H^m_D(X,\Z(n))[\ell^r] \to  H^m_B(X,\Z(n))[\ell^r]
 \end{align*}
 are surjective. Also the image of the composite of the right vertical arrows is equal to the image of $c^{m,n}_L$ restricted to $\ell^r$-torsion elements.
\end{remark}

\begin{lemma}\label{lemma}
Let $X$ be a smooth projective variety over $\C$, and fix an integer $k$ such that $1\leq k \leq \text{dim}_\C(X)$. Then the induced map of torsion groups $\Phi_X^k\Big|_\text{tors}:\left(\text{CH}^{k}_L(X)_{\text{hom}}\right)_{\text{tors}} \to J^k(X)_{\text{tors}}$ is an isomorphism.
\end{lemma}
\begin{proof}
Direct consequence of \cite[Proposition 5.1(a)]{RS}.
\end{proof}

\begin{remark}\label{rem}
Notice that if we set $k=\text{dim}(X)$, by Proposition \ref{lemGe} Chow groups and Lichntebaum cohomology coincide. Then we recover the classical Roitman's theorem
    \begin{align*}
        \CH^\text{hom}_0(X)_\text{tors} \simeq \text{Alb}_X(\C)_\text{tors}.
    \end{align*}
\end{remark}

We say that $W \subset H^{2k}_B(X,\Z(k))$ is a sub-Hodge structure if $W$ is s sub-lattice of $ H^{2k}_B(X,\Z(k))$ such that it has an induced Hodge decomposition $W_\C =\bigoplus_{p+q=2k} W^{p,q}$ with $W^{p,q}=W_\C \cap H^{p,q}$. Let $W\subset H^{2k}_B(X,\Z(k))$ be sub-Hodge structure, we define the partial Hodge conjecture with rational coefficients related to $W$ as the following statement: for every element $\alpha \in W$ there exists $N \in \mathbb{N}$ and an algebraic cycle $\widetilde{\alpha} \in \CH^k(X)$ such that $c(\widetilde{\alpha})=N\alpha$. It is clear that for $W=\text{Hdg}^{2k}(X,\Z)$ we recover the usual Hodge conjecture. For a fixed $W$ we denote the previous statement by $\text{HC}^k(X,W)_\Q$.

Similarly we denote by $\text{HC}^k_L(X,W)_{\Z}$ the statement that for every element of $\alpha \in W$ there exists a Lichtenbaum cycle $\widetilde{\alpha} \in \CH^k_L(X)$ such that $c_L(\widetilde{\alpha})=\alpha$. Then, inspired by the proof of \cite[Theorem 1.1]{RS}, we obtain the following result:

\begin{prop}\label{propE}
Let $X$ be a complex smooth projective variety and let $W \subset H^{2k}_B(X,\Z(k))$ be a sub-Hodge structure. Then $\text{HC}^k_L(X,W)_{\Z}$ holds  if and only if $ \text{HC}^k(X,W)_\Q$ holds.
\end{prop}
\begin{proof}
Let $W\subset \text{H}^{2k}_B(X,\Z(k))$ be a sub-Hodge structure and let $c^k_L:\text{CH}^k_L(X)\to \text{H}^{2k}_B(X,\Z(k))$ be the Lichtenbaum cycle class map constructed in \cite{RS} (similarly we can consider the classical cycle class map $c^k:\text{CH}^k(X)\to \text{H}^{2k}_B(X,\Z)$). Set $\text{CH}^k_{W,L}(X):=(c_L^k)^{-1}(W)$ as the preimage of $W$ in $\text{CH}_L^k(X)$. It is easy to see that $\text{CH}^k_L(X)_{\text{hom}} \hookrightarrow \text{CH}^k_{W,L}(X)$. Following with this notation, we will denote $I_{W,L}^k(X):=\text{im}(c_L^k)\cap W$, therefore $W$ is Lichtenbaum algebraic if and only if $Z_{W,L}^k(X):=W/I_{W,L}^k(X)=0$. In   the classical case, this is equivalent to say that $W_\Q$ is algebraic if and only if $Z^k_W(X)$ is a finite group. Since $I_W^k(X)\subset I_{W,L}^k(X)$ we have an exact sequence
\begin{align*}
0 \to I_{W,L}^k(X)/I_{W}^k(X) \to Z^k_W(X)\to Z^k_{W,L}(X)\to 0
\end{align*}

Denote $\pi:H^{2k}_D(X,\Z(k))\to \text{Hdg}^{2k}(X,\Z)$ the surjective map coming from the short exact sequence of Deligne-Beilinson cohomology, intermediate Jacobian and Hodge classes and denote $H^{2k}_{W,D}(X,\Z(k)):=\pi^{-1}(W)$. Then we have the following commutative diagram:

  \[
  \begin{tikzcd}
  0 \arrow{r} & \text{CH}^{k}_L(X)_{\text{hom}}  \arrow{r} \arrow{d}{c^k_{D,L}|_{\text{hom}}} &\text{CH}^k_{W,L}(X) \arrow{r}{c_L^k}  \arrow{d}{c^k_{D,L}|_{W_L^{-1}}} & I_{W,L}^k(X)  \arrow{d}{\text{into}}\arrow{r} & 0\\
  0 \arrow{r} &J^k(X) \arrow{r} & H^{2k}_{W,D}(X,\Z(k)) \arrow{r}   & W \arrow{r}& 0
  \end{tikzcd}
\]

Since $\text{CH}^k_L(X)_\text{hom}\otimes \Q/\Z =0$ by {\cite[Proposition 5.1 (b)]{RS}} and $J^k(X)$ is divisible, then we obtain the commutative diagram but with the torsion part of each group
  \[
  \begin{tikzcd}
  0 \arrow{r} & \left(\text{CH}^{k}_L(X)_\text{hom}\right)_{\text{tors}} \arrow{r} \arrow{d}{c^k_{D,L}|_\text{hom}} & \left(\text{CH}^k_{W,L}(X) \arrow{r}{c_L^k}\right)_{\text{tors}}  \arrow{d}{c^k_{D,L}|_{W_L^{-1}}} & I_{W,L}^k(X)_{\text{tors}}  \arrow{d}{\text{into}}\arrow{r} & 0\\
  0 \arrow{r} & J^k(X)_{\text{tors}} \arrow{r} & H^{2k}_{W,D}(X,\Z(k))_{\text{tors}} \arrow{r}   & W_{\text{tors}} \arrow{r}& 0
  \end{tikzcd}
\]
Due to the surjectivity of $\text{CH}^k_L(X)_{\text{tors}}\to H^{2k}(X,\Z)_{\text{tors}}$, the right vertical arrow is an isomorphism. If we can prove that the arrow in the middle is surjective, then we can conclude with similar arguments as in \cite{RS}, but this comes from the fact that  $\left(\text{CH}^k_L(X)_\text{hom}\right)_{\text{tors}}\simeq J^k(X)_{\text{tors}}$ by Lemma \ref{lemma}, and therefore  $c^k_{D,L}|_\text{hom}$ induces an isomorphism in the torsion part.

 Since we have an isomorphism $ \left(\text{CH}^k_{W,L}(X)\right)_{\text{tors}} \simeq H^{2k}_{W,D}(X,\Z)_{\text{tors}}$, we obtain a commutative diagram
\begin{equation}\label{6tor}
  \begin{tikzcd}
  0 \arrow{r} & A_{\text{tors}} \arrow{r} \arrow{d}{\simeq} & A \arrow{r}  \arrow{d}{c^k_{D,L}|_{A}} & A \otimes \Q  \arrow{d}\arrow{r} & A \otimes \Q/\Z \arrow{d}{\text{into}}\arrow{r}&0\\
    0 \arrow{r} & B_{\text{tors}} \arrow{r} & B \arrow{r}   & B \otimes \Q \arrow{r} & B\otimes \Q/\Z \arrow{r}& 0
  \end{tikzcd}
\end{equation}
where $A=\text{CH}^k_{W,L}(X) $ and $B=H^{2k}_{W,D}(X,\Z(k))$ and $A \otimes \Q/\Z \hookrightarrow B \otimes \Q/\Z$ is an injection, this can be seen in the computations done in Proposition \ref{lemma}. We can split diagram (\ref{6tor}) into two diagrams with short exact sequences as rows:
  \[
  \begin{tikzcd}
  0 \arrow{r} & A_{\text{tors}} \arrow{r} \arrow{d}{\simeq} & A \arrow{r}  \arrow{d}{c^k_{D,L}} & A_{\text{free}}  \arrow{d}{f}\arrow{r} & 0\\
    0 \arrow{r} & B_{\text{tors}} \arrow{r} & B \arrow{r}   & B_{\text{free}} \arrow{r} &  0
  \end{tikzcd}
\]
and
\begin{equation}\label{61tor}
  \begin{tikzcd}
  0 \arrow{r} & A_{\text{free}} \arrow{r} \arrow{d}{f} & A\otimes \Q \arrow{r}  \arrow{d}{c^k_{D,L}} & A\otimes \Q/\Z \arrow{d}{\text{into}}\arrow{r} & 0\\
    0 \arrow{r} & B_{\text{free}} \arrow{r} & B\otimes \Q \arrow{r}   & B\otimes \Q/\Z \arrow{r} &  0
  \end{tikzcd}
\end{equation}

The cokernel from the induced map $A\otimes \Q \to B\otimes \Q$ is torsion free as a quotient of $\Q$-vector spaces. Thus from diagram (\ref{61tor}), we obtain that $\text{coker}(f)$ is torsion free because it injects into a torsion free group, which, implies that $\text{coker}(c^k_{D,L}|_{A})$ is torsion free and,  along with the divisibility of $J^k(X)$, so $Z^k_{W,L}(X)$ as well.

The remaining part of the proof consists in proving that $I_{W,L}^k(X)/I^k_{W}(X)$ is a torsion group, but this comes from the fact that $I_{W,L}^k(X)$ and $I^k_{W}(X)$ have the same $\Z-$rank and therefore the quotient should be a finite group, so 
\begin{align*}
   Z^k_{W}(X) \otimes \Q=0 \iff Z^k_{W,L}(X)\otimes \Q=0\iff Z^k_{W,L}(X)=0.
\end{align*}
\end{proof}

\begin{remark}
Consider $X \in\text{SmProj}_\C$ of dimension $d$ and let us consider $H^{2d}_B(X\times X,\Z)$ modulo torsion. As we consider it modulo torsion, we apply the Künneth $H^{2d}_B(X\times X,\Z)\simeq \bigoplus H^{2d-i}_B(X,\Z)\otimes H^{i}_B(X,\Z)$ and let $\Delta_i \in H^{2d-i}_B(X,\Z)\otimes H^{i}_B(X,\Z)$ be the $i$-th component of the diagonal.  Consider $W_i$ be the sub-Hodge structure generated by $\Delta_i$, thus by Proposition \ref{propE} $W_i$ is $L$-algebraic if and only if $W_i\otimes \Q$ is algebraic, therefore the rational Künneth conjecture for $X$ holds if and only if the Künneth components are $L$-algebraic.
\end{remark}

\subsection{Generalized Hodge conjecture and Lichtenbaum cohomology}

Let $H$ be a pure Hodge structure of weight $n$ and let $0 \neq H_\C= H \otimes \C=\bigoplus_{p+q=n}H^{p,q}$. We say that $H$ is effective if and only if $H^{p,q}=0$ for $p< 0$ or $q<0$. The level of $l$ of $H$ is defined as $l = \text{max}\left\{|p-q| \ | \ H^{p,q}\neq 0\right\}$. Let $X$ be a smooth projective complex variety, we write $\text{GHC}(n,c,X)_\Q$ for the generalized Hodge conjecture in weight $n$ and level $n-2c$, where the conjectured result is the following:

\begin{conj}[{\cite[Generalized Hodge conjecture]{Gro}}]\label{conj1}
For every $\Q-$sub-Hodge structure $H\subset H^n(X,\Q)$ of level $n-2c$ there exists a subvariety $Y \subset X$ of pure codimension $c$ such that $H$ is supported on $Y$, i.e. $H$ is contained in the image of 
\begin{align*}
    H \subset \text{im}\left\{ H^l(\widetilde{Y},\Q(-c)) \xrightarrow{\gamma_*}H^n(X,\Q) \right\}
\end{align*}
where $\gamma_*=i_*\circ d_*$, $i_*$ is the Gysin map associated to the inclusion $i:Y\hookrightarrow X$ and $d:\widetilde{Y}\to Y$ is a resolution of singularities.
\end{conj}

There is an equivalent assertion of the generalized Hodge conjectures, in terms of algebraic cycles, for a proof of which, we refer to {\cite[Lemma 0.1]{Sc}}. 

\begin{conj}\label{conj2}
If $H\subset H^{n}(X,\Q)$ is a $\Q$-sub-Hodge structure of level $l=n-2c$, then $\text{GHC}(n,c,X)$ holds for $H$ if and only if there exists a smooth projective complex variety $Y$ together with an element $z \in \text{Corr}^c(Y,X)$ such that $H$ is contained in $z_* H^l(Y,\Q)$, where $z_*$ is given by the formula
\begin{align*}
z_*(\eta)=\text{pr}_{X*}\left(\text{pr}^*_Y(\eta)\cup c(z)\right).
\end{align*} 
\end{conj}
Also notice that this conjecture, and similar to the Hodge conjecture, can be stated in term of classical motives over $\C$
\begin{prop}{\cite[Page 301]{Gro}}
The generalized Hodge conjecture  for all $X \in \text{SmProj}_\C$ is equivalent to the following statement: the Hodge conjecture holds and an homological motive is effective if and only if its Hodge realization is effective.
\end{prop}

Based on the previous reformulation of the generalized Hodge conjecture, the authors Rosenschon and Srinivas proposed the following variant of the generalized Hodge conjecture for integral coefficients, but using Lichtenbaum cohomology groups:

\begin{conj}[L-Generalized Hodge Conjecture]\label{conjL}
Let $X$ be a smooth projective complex variety. If $H\subset H^{n}(X,\Z)$ is a $\Z$-sub-Hodge structure of level $l=n-2c$, then $\text{GHC}_L(n,c,X)$ holds for $H$ if and only if there exist a smooth projective complex variety $Y$ together with an element $z \in \text{Corr}^c_L(Y,X)$ such that $H$ is contained in $z_* H^l(Y,\Z)$.
\end{conj}

For a smooth projective complex variety $X$ the conjecture \ref{conjL} is denoted by $\text{GHC}_L(n,c,X)_\Q$. In some particular cases it is known to be equivalent to $\text{GHC}(n,c,X)_\Q$. For instance if we consider $\text{GHC}(2k-1,k-1,X)_\Q$  in {\cite[$\S 2$]{Gro}} it was mentioned that with this level and weights is related with the usual Hodge conjecture:

\begin{prop}{\cite[Remark 12.30]{Lew}}\label{prop}
Let $X$ be a smooth projective complex variety, then $\text{GHC}(2k-1,k-1,X)_\Q$ holds if and only if $\left( H^{2k-1}(X,\Q)\otimes H^1(\Gamma,\Q) \right)\cap H^{k,k}(\Gamma \times X)$ is algebraic for every smooth projective complex curve $\Gamma$. 
\end{prop}

The Lichtenbaum version of the previous result still holds as is stated in \cite[Remarks 5.2]{RS} whose proof uses similar arguments as the ones presented in \cite[Remark 12.30]{Lew}. Before we go into the proof of the proposition it is necessary to introduce some notations and conventions. First Betti cohomology is considered modulo torsion. Denote as
\begin{align*}
H_\text{L-alg}^{2k-1}(X,\Z)&:=\left\{\sigma_*:H^1(Y,\Z)\to H^{2k-1}(X,\Z) |
\ \sigma \in \text{Corr}_L^{k-1}(Y,X), \ \text{dim}Y=1\right\}/\text{tors}
\end{align*}
where $Y$ is smooth and projective, and recall
\begin{align*}
H_{\text{max}}^{2k-1}(X,\Z)&=\left\{\text{the largest } \Z\text{-sub HS in }\left\{H^{k,k-1}(X)\oplus H^{k-1,k}(X)\right\}\cap H^{2k-1}(X,\Z)\right\}
\end{align*}
the equality of both is equivalent to the generalized Hodge conjecture $\text{GHC}_L(2k-1,k-1,X)$. Note that $H^{2k-1}_\text{L-alg}(X,\Z)\otimes \C=H^{k,k-1}_\text{L-alg}(X)\oplus H^{k-1,k}_\text{L-alg}(X)$ because of the Hodge decomposition.  Also there exists a partial version of the previous result, which asks whether or not a sub-Hodge structure $W \subset H^{2k-1}(X,\Z)$ is contained in the image of the action of some Lichtenbaum correspondence over cohomology groups.

In the following proposition, we characterize this partial \'etale version of the generalized Hodge conjecture  of a Hodge structure of weight $2k-1$ and level 1, give a general description of the $\text{GHC}_L(2k-1,k-1,X)$ and its equivalence to $\text{GHC}(2k-1,k-1,X)_\Q$, as is stated in \cite[Remark 5.2]{RS}:

\begin{prop}\label{k}
Let $X \in \text{SmProj}_\C$, $k \in \mathbb{N}_{\geq 1}$ and let $W \subset H^{2k-1}(X,\Z)$ be a sub Hodge structure of level $1$. Then:
\begin{enumerate}
    \item[(i.)] there exist $Y\in  \text{SmProj}_\C$ and a Lichtenbaum correspondence $z \in \CH^{d_Y+1}_L(Y\times X)$ such that $W \subset z_* H^1(Y,\Z)$ if and only if for all curves $C \in  \text{SmProj}_\C$ the Hodge classes $H^{k,k}(C\times X)\cap \left\{H^1(C,\Z)\otimes W\right\}$ are algebraic.
    \item[(ii.)] In particular $\text{GHC}_L(2k-1,k-1,X)$ holds if and only if for all curves $C \in  \text{SmProj}_\C$ the Hodge classes $H^{k,k}(C\times X)\cap \left\{H^1(C,\Z)\otimes H^{2k-1}(X,\Z)\right\}$ are L-algebraic, i.e. 
    $$H^{k,k}(C\times X)\cap \left\{H^1(C,\Z)\otimes H^{2k-1}(X,\Z)\right\} \subset \text{im}\left\{c_L^k:\CH^k_L(C\times X)\to H^{2k}_B(X,\Z(k))\right\}.$$
    \item[(iii.)] $\text{GHC}(2k-1,k-1,X)_\Q$ holds if and only if $\text{GHC}_L(2k-1,k-1,X)$ holds.
\end{enumerate}

\end{prop}

\begin{proof}
For (i.), let $W \subset H^{2k-1}(X,\Z)$ be a sub Hodge structure of weight $2k-1$ and level $1$ and assume that there exists $Y\in  \text{SmProj}_\C$ and a Lichtenbaum correspondence $z \in \CH^{d_Y+k-1}_L(Y\times X)$ (correspondence of degree $k-1$) such that $W \subset z_* H^1(Y,\Z)$. Consider $C$ a smooth complex projective curve and consider an element $h \in \left\{H^1(Y,\Z)\otimes W\right\}\cap H^{k,k}(C\times X)\simeq \text{Hom}_{\text{HS}}(H^1(C,\Z),W)$. Let $h_*: H^1(C,\Z)\to W$ be the map of Hodge structures induced by $h$. Define $V=\text{ker}(z_*)$, which by the theory of Hodge structures, is a Hodge structure itself. We know that the image of $\text{im}(z_*)$ is a  Hodge structures of the same weight (see {\cite[Lemme 7.23 et 7.25]{Voi}}), then $H^1(Y,\Z)=V\oplus R$ where $R= V^\perp$. Then we have a morphism $\lambda:=\left(z_*|_R\right)^{-1}\Big|_W: \text{im}(z_*)\cap W \subset H^{2k-1}(X,\Z)\to R$ which fits into the following commutative diagram
  \[
  \begin{tikzcd}
H^1(C,\Z) \arrow{r}{h_*} \arrow{d}{h_*} & W \arrow[equal]{r} & \text{im}(z_*)\cap W \arrow{r}{\lambda} & H^1(Y,\Z) \arrow{d}{z_*} \\
W \arrow[equal]{rrr} & & & W.
  \end{tikzcd}
\]
The map obtained obtained from the upper arrows is induced by a Hodge class by Lefschetz (1,1) and therefore $h$ is algebraic. 

Conversely, suppose that for all smooth and projective curve $C$ the Hodge classes $H^{k,k}(C\times X)\cap \left\{H^1(C,\Z)\otimes W\right\}$ are algebraic. Let $W\subset H^{2k-1}(X,\Z)$ be a sub-Hodge structure of level 1 and notice that $W$ has a decomposition as $W\otimes \C=W^{k,k-1}\oplus W^{k-1,k}$. Then its associated $k$-th intermediate Jacobian is of the form $J^k(W)=W^{k-1,k}/W$ which is an abelian variety. Since $J^k(W)$ is a complex torus, then its holomorphic tangent bundle is $W^{k-1,1}$ and the fundamental group is isomorphic to the lattice $W$, thus $\pi_1(J^k(W))\simeq H_1(J^k(W),\Z)=W$. Set $m=\text{dim}(J^k(W))$ then $H^{2m-1}(J^k(W),\C)=H^{m-1,m}(J^k(W))\oplus H^{m,m-1}(J^k(W))$ and
\begin{align*}
H^{m-1,m}(J^k(W))&\simeq H^{1,0}(J^k(W))^* \\
&=H^0(J^k(W),\Omega^1_{J^k(W)})^*\\
&\simeq H^0(J^k(W),\Omega^1_{J^k(W)})^*\\
&\simeq H^0(J^k(W),T^*_{J^k(W)})^*\simeq W^{k-1,k}.
\end{align*}
then $H^{2m-1}(J^k(W),\C)=W^{k-1,k}\oplus W^{k,k-1}$.

 Taking hyperplane sections of $J^k(W)$ and applying Bertini's theorem, we find a smooth projective curve $\Gamma \subset J^k(W)$ and a surjective map $H_1(\Gamma,\Z)\to H_1(J^k(W),\Z)\simeq W$. Also by Poincar\'e duality $H_1(\Gamma,\Z)\simeq H^1(\Gamma,\Z)$ so we have a surjective map $f: H^1(\Gamma,\Z) \to W$. Since the map $f$ is a morphism of Hodge structures, then it is an element in $H^{k,k}(\Gamma\times X)\cap \left\{H^1(\Gamma,\Z)\otimes W\right\}$ which by hypothesis is L-algebraic. Therefore there exists a class $z \in \CH^{2k}_L(\Gamma \times X)$ such that $W \subset  z_* H^1(\Gamma,\Z) \subset H^{2k-1}(X,\Z)$.
 
The statement (ii.) is a direct consequence of (i.) taking $W=H^{2k-1}(X,\Z)$ and the maximal sub Hodge structure of it $H^{2k-1}_\text{max}(X,\Z)$. For (iii.) notice that for a complex smooth projective curve $C$ the Betti cohomology groups are torsion free. Thus Künneth formula holds for the product $C\times X$ and then
\begin{align*}
    H^{k,k}(C\times X)\cap \left\{H^1(C,\Z)\otimes H^{2k-1}(X,\Z)\right\} \subset  H^{k,k}(C\times X)\cap H^{2k}(C\times X,\Z)=\text{Hdg}^{2k}(C\times X,\Z).
\end{align*}
Invoking Proposition \ref{propE} $H^{k,k}(C\times X)\cap \left\{H^1(C,\Z)\otimes H^{2k-1}(X,\Z)\right\}$ is L-algebraic if and only if $H^{k,k}(C\times X)\cap \left\{H^1(C,\Q)\otimes H^{2k-1}(X,\Q)\right\}$ is algebraic in the usual sense, which gives us the equivalences
 \begin{align*}
     \text{GHC}(2k-1,k-1,X)_\Q&\text{ holds} \\ &\iff \ H^{k,k}(C\times X)\cap \left\{H^1(C,\Q)\otimes H^{2k-1}(X,\Q)\right\} \text{ is alg. $\forall$ curve }C \\
     &\iff H^{k,k}(C\times X)\cap \left\{H^1(C,\Z)\otimes H^{2k-1}(X,\Z)\right\} \text{ is L-alg. $\forall$ curve }C \\
     &\iff  \text{GHC}_L(2k-1,k-1,X)\text{ holds.}
 \end{align*}
\end{proof}

In the sequel, we give more subtle relations between the Hodge conjecture and the generalized one, following the proof of the classical case given in \cite[Lemma 2.3]{Fu}:

\begin{lemma}
Let $X$ be a smooth projective variety of dimension $n$ and $H\subset H^k(X,\Z)$ be a sub-Hodge structure of coniveau at least $c$ and assume that there exists a smooth projective variety $Y$ of dimension $d_Y$, such that $H(c)$ is a sub-Hodge structure of $H^{k-2c}(Y,\Z)$. Then if $H^{d_Y+c,d_Y+c}(Y\times X)\cap \left\{H^{2(d_Y+c)-k}(Y,\Z)\otimes H^{k}(X,\Z)\right\}$ is L-algebraic then generalized L-Hodge conjecture for $H$ holds. 
\end{lemma}
\begin{proof}
Since torsion classes come from Lichtenbaum cycles, for simplicity we will neglect torsion Hodge classes. Suppose that $H$ is a sub-Hodge structure of $H^k(X,\Z)$ of weight $k$ and coniveau $c$. We know that $H(c)$ is still an effective Hodge structure, then there is a smooth projective variety $Y$ such that $H(c)$ is a sub-Hodge structure of $H^{k-2c}(Y,\Z)$, which by polarization can be decomposed as $H^{k-2c}(Y,\Z)\simeq H(c)\oplus R$. Consider $f:H^{k-2c}(Y,\Z)\to H^{k}(X,\Z)$ the morphism resulting from the composition of the following maps
\begin{align*}
    H^{k-2c}(Y,\Z)\xrightarrow{\text{pr}_1} H(c) \xrightarrow{\text{id}\otimes \Z(-c)} H \hookrightarrow H^{k}(X,\Z)
\end{align*}
 Since $ \ho_{\text{HS}\Z}(H^{k-2c}(Y,\Z),H^{k}(X,\Z)) \simeq H^{2(d_Y+c)-k}(Y,\Z)\otimes H^k(X,\Z) $ by the hypothesis that $H^{d_Y+c,d_Y+c}(Y\times X)\cap \left\{H^{2(d_Y+c)-k}(Y,\Z)\otimes H^{k}(X,\Z)\right\}$ is L-algebraic we conclude that $f$ comes from a Lichtenbaum algebraic cycle $\gamma \in \CH^{d_Y+c}_L(Y\times X)$ and $H \subset \gamma_* H^{k-2c}(Y,\Z)$. Thus the generalized L-Hodge conjecture holds for $H$. 
\end{proof}

Using the same kind of arguments, and adding an hypothesis of effectiveness it is possible to characterize the generalized Hodge conjecture in terms of the integral Hodge conjecture in the \'etale setting.
 
 \begin{theorem}\label{teo}
The Lichtenbaum generalized Hodge conjecture  for all $X \in \text{SmProj}_\C$ holds if and only if the following two conditions hold: 
\begin{itemize}
    \item the Lichtenbaum Hodge conjecture holds,
    \item a homological \'etale motive is effective if and only if its Hodge realization is effective.
\end{itemize}
\end{theorem}

\begin{proof}
The generalized L-Hodge conjecture immediately implies the L-Hodge conjecture. Suppose that $M$ has an effective realization and let $H:=\rho_H(M)$ be its associated Hodge structure of weight $n$ and coniveau $c$. By the generalized L-Hodge conjecture there exists $Y\in \text{SmProj}_\C$ and $\gamma \in \CH^{d_Y+c}_L(Y\times X)$ such that $H \subset \gamma_* H^{n-2c}(Y,\Z) \subset H^n(X,\Z)$. The motive $M(c)$ is effective in $h(Y)$, thus $M$ is effective because it is a subobject of the effective motive $h(X)$.

Assume that the L-Hodge conjecture holds for every $X \in \text{SmProj}_\C$ and that a homological motive is effective if and only if its realization is effective. We can neglect torsion Hodge classes because they always come from torsion algebraic cycles. Suppose that $H$ is a sub-Hodge structure of $H^n(X,\Z)$ of weight $n$ and coniveau $c$. We know that $H(c)$ is still an effective Hodge structure. Then there is a smooth projective variety $Y$ such that $H(c)$ is a sub-Hodge structure of $H^{n-2c}(Y,\Z)$ which by polarization can be decomposed as $H^{n-2c}(Y,\Z)\simeq H(c)\oplus R$. Consider $f:H^{n-2c}(Y,\Z)\to H^{n}(X,\Z)$ the morphism resulting from the composition of the following maps
\begin{align*}
    H^{n-2c}(Y,\Z)\xrightarrow{\text{pr}_1} H(c) \xrightarrow{\text{id}\otimes \Z(-c)} H \hookrightarrow H^{n}(X,\Z)
\end{align*}
 Since $ \ho_{\text{HS}\Z}(H^{n-2c}(Y,\Z),H^{n}(X,\Z)) \simeq H^{2(d_Y+c)-n}(Y,\Z)\otimes H^n(X,\Z)$ by the assumption of the Hodge conjecture and Proposition \ref{propE} we conclude that $f$ comes from a Lichtenbaum algebraic cycle $\gamma \in \CH^{d_Y+c}_L(Y\times X)$ and $H \subset \gamma_* H^{n-2c}(Y,\Z)$. Thus the generalized L-Hodge conjecture holds. 
\end{proof}

Then we have the following corollary coming from the previous characterizations of the Generalized Hodge conjecture (classical and Lichtenbaum setting)
\begin{corollary}\label{corf}
The generalized Hodge conjecture holds if and only if the generalized L-Hodge conjecture holds.
\end{corollary}

\subsection{Bardelli's example}

Let us recall the example presented in \cite{Bar} of a certain threefold $X$ where $\text{GHC}(3,1,X)_\Q$ holds. Let $\sigma: \Pro^7\to \Pro^7$ be the involution defined as $\sigma(x_0:\ldots:x_3:y_0,\ldots,y_3)=(x_0:\ldots:x_3:-y_0,\ldots,-y_3)$ 
 and let $X=V(Q_0,Q_1,Q_2,Q_3)$ be a smooth complete intersection of four $\sigma$-invariant quadrics. There exists a smooth irreducible curve $C$, of genus 33,  obtained as the intersection of two nodal surfaces, and an \'etale double covering $\widetilde{C}\to C$ such that $H^1(\widetilde{C},\Q)^- \to H^3(X,\Q)^-$ is surjective, where the first group is the anti-invariant part of the involution $\tau: \widetilde{C}\to \widetilde{C}$ associated to the double covering and the later group is the anti-invariant part associated to the involution $\sigma$. Notice that by \cite[Fact 2.4.1]{Bar} if we assume that  $X$ is a very general threefold, then $H^3(X,\Q)^+$ and $H^3(X,\Q)^-$ are perpendicular with respect to the cup product on $H^3(X,\Q)$ and $H^{3,0}(X) \subset H^3(X,\C)^+$ therefore $H^3(X,\Q)^-$ is a polarized Hodge structure perpendicular to $H^{3,0}(X)$ i.e. a polarized sub-Hodge structure of $H^3(X,\Q)$ of level 1. The isogeny $\alpha: \text{Prym}(\widetilde{C}\to C)\to J(X)^-$, where $ J(X)^-$ is the projection of $H^{1,2}(X)^-$ into $J^2(X)$, is the correspondence that induces the isomorphism $H^1(\widetilde{C},\Q)^- \to H^3(X,\Q)^-$, but in the case of integral coefficients the image of the correspondence is a subgroup of index 2. From the previous results we have the following equivalences:
 \begin{align*}
     \text{GHC}(3,1,X)_\Q\text{ holds for }&H^3(X,\Q)^- \\ &\iff \ H^{2,2}(\Gamma \times X)\cap \left\{H^1(\Gamma,\Q)\otimes H^3(X,\Q)^-\right\} \text{ is alg. $\forall$ curve }\Gamma \\
     &\iff H^{2,2}(\Gamma\times X)\cap \left\{H^1(\Gamma,\Z)\otimes H^3(X,\Z)^-\right\} \text{ is L-alg. $\forall$ curve }\Gamma \\
     &\iff  \text{GHC}_L(3,1,X)\text{ holds for }H^3(X,\Z)^-
 \end{align*}
 so there exists a smooth projective curve $\Gamma'$ and a correspondence $z \in \CH^2_L(\Gamma'\times X)$ such that $H^3(X,\Z)^- \subset z_* H^1(\Gamma',\Z)$. 

\printbibliography[title={Bibliography}]
 
\info

\end{document}